\documentclass[12pt, reqno]{amsart}
\usepackage{amsmath}
\usepackage{amsfonts}
\usepackage{amssymb}
\usepackage[english]{babel}
\usepackage{color}

\usepackage{comment}
\usepackage{calrsfs}
\usepackage{accents}
\usepackage[utf8x]{inputenc}
\usepackage{hyperref}

\usepackage{bbding}
\usepackage[normalem]{ulem}

\theoremstyle{plain}
\theoremstyle{definition}
\newtheorem {theorem}{Theorem}[section]
\newtheorem {lemma}[theorem]{Lemma}
\newtheorem {corollary} [theorem]{Corollary}

\newtheorem{definition}[theorem]{Definition}
\newtheorem{remark}[theorem]{Remark}

\theoremstyle{remark}

\long\def\MSC#1\EndMSC{\def\arg{#1}\ifx\arg\empty\relax\else
     {\par\narrower\noindent
     {\small\it 2010 Mathematics Subject Classification.} \small #1\par}\fi}

\long\def\KEY#1\EndKEY{\def\arg{#1}\ifx\arg\empty\relax\else
     {\par\narrower\noindent
     {\small\it Keywords and Phrases.} \small #1\par}\fi}

\newcommand{\R}{\mathbb R}
\newcommand{\N}{\mathbb N} 

\newcommand{\Tan}{\text{Tan}}
\newcommand{\Exc}{\text{Exc}}
\newcommand{\g}{\mathfrak{g}}
\newcommand{\obar}[1]{\overline{#1}}
\newcommand{\opi}{\obar{\pi}}

\DeclareMathOperator{\ad}{ad}
\DeclareMathOperator{\Ad}{Ad}
\newcommand{\abs}[1]{\left|#1\right|}
\newcommand{\barint}
{\rule[.036in]{.12in}{.009in}\kern-.16in \displaystyle\int}

\usepackage{enumitem}
\newcommand\restr[2]{{
  \left.\kern-\nulldelimiterspace
  #1
  \vphantom{\big|}
  \right|_{#2}
  }}
\newcommand{\ang}[1]{\left\langle#1\right\rangle}
\DeclareMathOperator{\dimens}{dim}
\newcommand{\ubar}[1]{\underline{#1}}
\newcommand{\set}[1]{\left\{#1\right\}}
\newcommand{\pa}[1]{\left(#1\right)}
\newcommand{\bra}[1]{\left[#1\right]}
\newcommand{\norm}[1]{\left\|#1\right\|}
\usepackage{breqn}
\newcommand{\leb}{\mathcal{L}}
\newcommand{\G}{G}
\newcommand{\dgamma}{\dot{\gamma}}
\newcommand{\ugamma}{\underline{\gamma}}
\newcommand{\incr}[1]{\left. #1 \right|}
\DeclareMathOperator{\exc}{Exc}
\DeclareMathOperator{\dev}{Cor}
\DeclareMathOperator{\cut}{Cut}

\newcommand{\media}{\barint}

\newcommand{\duebeta}{\beta}

\renewcommand{\epsilon} {\varepsilon}
\renewcommand{\phi} {\varphi}
\renewcommand{\bar}{\overline}

\subjclass[2010]{53C17, 49K30, 28A75}
\keywords{Geodesics, regularity of length minimizers, Carnot--Carathéodory spaces, sub-Riemannian geometry, Carnot groups}

\begin{document}


\author[R.~Monti]{Roberto Monti}
\email{monti@math.unipd.it}

\author[A.~Pigati]{Alessandro Pigati}
\email{alessandro.pigati@math.ethz.ch}

\author[D.~Vittone]{Davide Vittone}
\email{vittone@math.unipd.it}

\address[Monti and Vittone]
{Universit\`a di Padova, Dipartimento di Matematica,
via Trieste 63, 35121 Padova, Italy}

\address[Pigati]
{Scuola Normale Superiore, Piazza dei Cavalieri 7, 56126 Pisa, Italy}

\address[Pigati]{ETH Z\"urich, Department of Mathematics,
R\"amistrasse 101, 8092 Z\"urich, Switzerland}

\thanks{R.~M. and D.~V. are supported by MIUR (Italy) and University of Padova. D.~V. is supported by University of Padova Project Networking and INdAM-GNAMPA Project 2017 ``Campi vettoriali, superfici e perimetri in geometrie singolari''.}

\title[Tangent lines to geodesics]{Existence of tangent lines to Carnot--Carath\'eodory geodesics}

\begin{abstract}
We show that length minimizing curves in Carnot--Carath\'eodory
spaces possess at any point at least one tangent curve (i.e., a blow-up in the
nilpotent approximation)  equal to a straight horizontal line. This is the
first regularity result for length minimizers that holds with no assumption
on either the space (e.g., its rank, step, or analyticity) or the curve, and it is novel even in the  setting of Carnot groups.
\end{abstract}

\maketitle



\section{Introduction}
Let $M$ be a connected $n$-dimensional $C^\infty$-smooth manifold and  $\mathcal X=\{ X_1,\ldots,X_r\}$, $r\geq 2$,  a system of linearly independent $C^\infty$-smooth vector fields on $M$ satisfying the H\"ormander  condition.
We call the pair $(M,\mathcal X)$ a \emph{Carnot--Carath\'eodory (CC) structure}.
 Given an interval   $I\subseteq\R$, a Lipschitz curve $\gamma:I\to M$ is
said to be \emph{horizontal}
if there exist  functions $h_1,\ldots, h_r \in L^\infty(I)$ such that for
a.e. $t\in I$ we have
\begin{equation}\label{horiz}
	\dot \gamma (t) = \sum_{i=1}^r h_i(t) X_i (\gamma(t)).
\end{equation}
Letting $|h|:=(h_1^2+\ldots+h_r^2)^{1/2}$, the length of $\gamma$ is then
defined as
\[
	L(\gamma):=\int_I |h(t)|\,dt.
\] 
We will usually  
assume that curves are parameterized by
arclength, i.e., $|h(t)|=1$ for a.e.~$t$, so that $\mathcal L^1(I) = L(\gamma)$.

Since $M$ is connected, by Chow--Rashevsky theorem for any pair of
points $x,y \in M$ there exists a horizontal curve
joining $x$ to $y$. We can therefore define a distance function $d: M\times
M\to[0,\infty)$ letting
\begin{equation}\label{diddi}
	d(x,y) := \inf \big\{ L(\gamma) \mid \gamma: [0,T]\to M \textrm{ horizontal with
	$\gamma(0) = x$ and $\gamma(T)=y$}\big\}.
\end{equation}
The resulting metric space $(M,d)$ is a \emph{Carnot--Carath\'eodory  space}.
Typical examples of Carnot--Carath\'eodory spaces are given by sub-Riemannian manifolds
$(M, \mathcal D, g)$, where
$\mathcal D \subset TM$ is a completely non-integrable distribution and $g$ is a
smooth metric on $\mathcal D$.

If the closure of any ball in $(M,d)$ is compact, then
the infimum in \eqref{diddi} is a minimum, i.e., any pair of points can be
connected by a length-minimizing curve.
A horizontal curve $\gamma:[0,T]\to M$ is a \emph{length minimizer} if $L(\gamma) =
d(\gamma(0),\gamma(T))$. In Carnot--Carath\'eodory spaces (or even  in the model
case of {\em Carnot groups})
it is not known whether constant-speed length minimizers are 
$C^\infty$-smooth, 
or even $C^1$-smooth.
The main obstacle is the presence of \emph{abnormal} 
length minimizers, which are not captured by the natural Hamiltonian framework,
see e.g.~\cite[Section~1.5]{montgomery}. In \cite{Mon94}, Montgomery gave the first
example of such a length minimizer. Contrary to the Riemannian case,
stationarity conditions do not guarantee any smoothness of the curve: in
\cite{LLMV13} it is proved that no further regularity beyond the Lipschitz one
can be obtained for abnormal extremals from the Pontryagin Maximum Principle and
  the Goh condition (which is a second-order necessary condition, see e.g. \cite[Chapter~20]{agrsach}). 

However, some partial regularity results are known. If the {\em step} is at most
$2$ (i.e., for any $x$ the tangent space  $T_xM$ is spanned by the $r+\binom{r}{2}$ vectors
$X_i(x)$, $[X_i,X_j](x)$), then all constant-speed length minimizers are smooth.
In the context of  Carnot groups, the regularity problem was recently 
solved also when the step is at most $3$ by Le Donne, Leonardi, Monti and Vittone
in \cite{LDLMV1}. In \cite{Sus14}  Sussmann proved that, in presence
of analytic data (and in particular in Carnot groups), all length minimizers are
\emph{analytic} on a dense open set of times, although it is not known whether
this set has full measure. Building on ideas contained in \cite{LM08,LLMVESAIM},
Hakavuori and Le Donne recently proved in \cite{HL16} that length minimizers
cannot have corner-type singularities. Other partial regularity results are contained in \cite{mo}. We also refer to \cite{AOP,montiregprob,R,Vit14} for surveys about the known results on the problem.

It is well-known that, at any point  $x\in
M$, the space $(M,\mathcal X)$ has a \emph{nilpotent approximation}
$(M^\infty,\mathcal X^\infty)$, which is itself a
Carnot--Carath\'eodory structure.
The corresponding metric space is obtained as a
pointed Gromov--Hausdorff limit of metric spaces. An elementary construction of
$(M^\infty,\mathcal X^\infty)$ is detailed in \cite{MPVconotangente}.

When $t\in (-T,T)$ is fixed and we perform
this construction for $x=\gamma(t)$, we denote by $\Tan(\gamma;t)$ the set
containing all possible  curves in $M^\infty$ that arise as limits of  $\gamma$ 
in the local uniform topology. The {\em tangent cone} $\Tan(\gamma;t)$ was
introduced in \cite{MPVconotangente}, where it was also proved that its elements
are length minimizing horizontal curves in $M^\infty$   parametrized by
arclength. We call {\em horizontal line} any horizontal curve in
$(M^\infty,\mathcal X^\infty)$ passing through the base point in $ M ^\infty$
and  with \emph {constant}   controls.

The following theorem is the main result of the paper.

\begin{theorem}  \label{1.1}
	Let $\gamma:[-T,T]\to M$ be a  length minimizer parametrized by arclength in a Carnot--Carath\'eodory
	space $(M,d)$. Then, for any $t\in(-T,T)$, the tangent cone $\Tan(\gamma;t)$ contains a horizontal
	line.
\end{theorem}

Theorem \ref{1.1} has an analytic reformulation,  stated solely in terms of the
control $h$, which is independent of the notion of nilpotent
approximation: see Remark \ref{contreform}.  A version of Theorem
\ref{1.1} holds for the extremal 
points $t=0$ and $t=T$ of a length minimizer  $\gamma:[0,T]\to M$. In this 
case, the tangent cone contains a horizontal half-line; see Theorem
\ref{thm:1.1onesided}.
These results imply and improve  the ones
contained in \cite{LM08,LLMVESAIM,HL16}: while in these papers the existence of
(linearly independent) left and right derivatives is assumed
in order to construct a shorter competitor, Theorem \ref{1.1} provides a mild
form of pointwise
differentiability which
automatically rules out corner-type singularities.

Theorem \ref{1.1} is deduced from a similar result for the case when $M=G$ is
a Carnot group of rank $r\geq 2$
and $\mathcal X = \{ X_1,\ldots, X_r\}$ is a system of left-invariant vector
fields forming a basis of the first layer of its Lie algebra $\g$. As explained in the proof of Theorem \ref{1.1}, the reduction to this case is made possible by the results proved in \cite{MPVconotangente}.

The proof in the case of a Carnot group, in turn, is a
consequence of Theorem \ref{1.2} below.
Let $ \mathfrak g= \mathfrak g_1\oplus \cdots \oplus \mathfrak g_{s}$ be the
stratification of $\mathfrak g$ 
and let $\langle \cdot,\cdot\rangle$ be the scalar product on $\mathfrak g_1$
making $X_1 ,\ldots,X_r $ orthonormal. 
The integer $s\geq 2$ is the step of the group and $r=\mathrm{dim}\: \g_1$ its rank. We denote by $S^{r-1} =\{ v\in
\mathfrak g_1:\langle v,v\rangle =1\}$ the unit sphere in $\mathfrak g_1$.
We define the \emph{excess} of a horizontal curve $\gamma:[-T,T]\to G$ over a Borel
set $B\subseteq [-T,T]$ with positive measure as
\[
	\Exc (\gamma;B) := \inf_{v\in S^{r-1}}\Big( \barint _B \langle
	v,\dot\gamma(t)\rangle ^2\,dt\Big)^{1/2}.
\]
The excess $\Exc (\gamma;B)$ measures how far $\dgamma_{|B}$ is from being contained in a single hyperplane of $\g_1$, see Remark \ref{rem:excpiani}. For length-minimizing curves, the excess is infinitesimal at suitably small
scales, as stated in our second main result.

\begin{theorem}\label{1.2}
	Let $G$ be a Carnot group and let $\gamma:[-T,T]\to G$, $T>0$, be a
	length-minimizing curve parametrized by arclength.
	Then there exists an infinitesimal sequence $\eta_i\downarrow 0$ 
	  such that
	\begin{equation} \label{ecc}
		\lim_{i\to\infty} \Exc(\gamma;[-\eta_i,\eta_i]) = 0.
	\end{equation}
\end{theorem}

Again, this result has a version for extremal points: for a length minimizer $\gamma:[0,T]\to G$ the excess
$\Exc(\gamma;[0,\eta_i])$ is infinitesimal, see Theorem \ref{onesthm}.
When $r=2$, \eqref{ecc} implies that there exists $\kappa \in
\Tan(\gamma;0)$ of the form
$\kappa(t) = \exp(tv)$ for some $v \in \mathfrak g_1$. This proves Theorem
\ref{1.1} for $M=G$ with $r=2$.
When $r>2$, the situation can be reduced by induction
to the case $r=2$ using, again, the outcomes of  \cite{MPVconotangente}.  

The introduction of the excess is probably among the main contributions of this
paper; the reader familiar with the regularity theory of minimal hypersurfaces
in $\R^n$ will notice the analogy with the quantity that plays a key role  
in De Giorgi's approach to that problem, see e.g.~\cite{giusti}, and in many
subsequent results inspired by his work (we just mention e.g. \cite{allard} and \cite{evans}). In this
sense, Theorems \ref{1.1} and \ref{1.2} constitute a first step towards a
regularity theory for length minimizers, whose next stages (height bounds,
Lipschitz approximation theorems, reverse Poincar\'e inequality, harmonic
approximation according to the terminology of \cite{maggi}) could now see their
way paved by Theorems \ref{1.1} and \ref{1.2}.

We conclude this introduction by spending a few words about the proof of Theorem \ref{1.2}. As  detailed in Section \ref{sec:proofs}, it goes by contradiction 
and uses a cut-and-adjust construction performed in $s$ steps. If we
had $\Exc(\gamma;[-\eta,\eta])\geq \varepsilon$ for some
$\varepsilon >0$ and for all small $\eta>0$, then we could find $t_1<\cdots<t_r$
such that, roughly speaking,
the vectors $\dot\gamma(t_1),\dots,\dot\gamma(t_r)\in\mathfrak g_1$ are
linearly independent in a quantitative way, see Lemma \ref{excrinc}. 
We could replace the ``horizontal projection'' $ \underline{\gamma}$ of $\gamma$
on
the interval $[-\eta,\eta]$ with the line segment joining
$\underline{\gamma}(-\eta)$
to $\underline{\gamma}(\eta)$, whose gain of length would be estimated from below in terms
of the
excess, see Lemma \ref{excgain}, and 
we could lift the resulting ``horizontal coordinates'' to a horizontal
curve in $G$.
The  end-point of the new curve might be different,
but the vectors $\dot\gamma(t_1),\dots,\dot\gamma(t_r)$ could then be used to build
suitable correction devices
restoring the end-point, taking care to keep a positive gain of length. This construction is detailed in Sections
\ref{sec:exc}, \ref{sec:cutcor} and \ref{sec:proofs} and  it is a refinement of the techniques introduced and developed
in \cite{LM08} and \cite{HL16}. In particular, Section \ref{sec:cutcor} 
contains
explicit formulas, for the length gain associated with the cut and for the
displacement of the final point caused by the application of the correction
devices, which make the constructions in \cite{LM08,HL16} more transparent. We
believe that these formulas have an independent interest and could possibly be
useful for future applications.

{\em Acknowledgements.} We thank L. Ambrosio for several discussions and for being an invaluable mentor and friend.

\section{Excess, compactness of length minimizers and first
consequences}\label{sec:exc}

In this section we prove Lemma \ref{excrinc}, which provides the correct position
for the correction devices introduced in Section \ref{sec:cutcor}. We work in
the setting of a Carnot group.

\begin{definition}
A {\em Carnot group} is a finite dimensional connected,  simply connected and
nilpotent Lie group $G$ whose Lie algebra $\g$ is {\em stratified}, i.e., there
exists a (fixed) decomposition $ \g= \g_1\oplus\cdots\oplus \g_s $ such that
$\g_j=[\g_1,\g_{j-1}]$ for any $j=2,\dots, s$ and $[\g,\g_s]=\{0\}$.
\end{definition}

We will denote by $n$ the dimension of $\g$ and by $r$ the dimension of its first layer $\g_1$; we  refer to the integers $r,s$ as the {\em rank} and {\em step} of $\g$, respectively. We  endow $ \g$ with a positive definite scalar
product $\ang{\cdot,\cdot}$ such that $\g_i\perp \g_j$ whenever $i\neq j$. We also let $|\cdot| :=\langle\cdot,\cdot\rangle^{1/2}$. We fix an orthonormal basis
$X_1,\dots,X_n$ of $\g$ adapted to the stratification, i.e.,  such that  
$\g_j=\mathrm{span}\{ X_{r_{j-1}+1},\dots,X_{r_j}\}$ for any $j=1,\dots,s$,
where  $r_j:=\mathrm{dim} (\g_1)+\dots+\mathrm{dim}(\g_j)$ and $r_0:=0$. 

For $\lambda>0$, the {\em dilations} $\delta_r:\g\to\g$ defined by
\[
\delta_\lambda (X):=\lambda^j X,\qquad\text{if }X\in \g_j,
\]
form a one-parameter group of isomorphisms of $\g$. Being $\g$ nilpotent, the exponential map $\exp:\g\to G$ is a diffeomorphism and, by composition with $\exp$, the dilations on $\g$ induce a one-parameter family of group isomorphisms, which we still denote by $\delta_\lambda:G\to G$. We recall for future reference the Baker--Campbell--Hausdorff formula: for
any $X,Y\in\g$ we have $\exp(X) \exp(Y)=\exp(P(X,Y))$, where
\begin{equation}\label{eq:bchpol}
P(X,Y):=\sum_{p=1}^s\frac{(-1)^{p+1}}{p}\sum_{1\le k_i+\ell_i\le s}\frac{[X^{k_1},Y^{\ell_1},\dots,X^{k_p},Y^{\ell_p}]}{k_1!\cdots k_p!\ell_1!\cdots\ell_p!\sum_i(k_i+\ell_i)}.
\end{equation}
Here, we use the short notation $[Z_1,\dots,Z_{k+1}]:=(\ad Z_1)\cdots(\ad
Z_k)Z_{k+1}$, with $\ad X:\g\to\g$ being the adjoint mapping $\ad X(Y) :=
[X,Y]$.

The group $G$ is endowed with the Carnot--Carath\'eodory distance $d$  induced
by the family $X_1,\dots, X_r$, for which one clearly has, for
$x,y,z\in G$ and $\lambda>0$,
\begin{equation}\label{simmetriedist}
d(zx,zy)=d(x,y)\quad\text{and}\quad
d(\delta_\lambda(x),\delta_\lambda(y))=\lambda d(x,y)
.
\end{equation}
We will frequently use the homogeneous (pseudo-)norm $\|x\|$ defined in this way: if  $x=\exp(Y_1+\dots+Y_s)$ for $Y_j\in\g_j$, then
\[
\| x\| := \sum_{j=1}^s |Y_j|^{1/j}.
\]
A well-known consequence of the homogeneity of $\|\cdot\|$ and \eqref{simmetriedist} is the fact that $\|\cdot\|$ is equivalent to the distance function from the identity 0 of $G$. In particular there exists $C>0$ such that, for any horizontal curve $\gamma:[-T,T]\to G$ parametrized by arclength and such that $\gamma(0)=0$, the estimate
\begin{equation}\label{rem:potenzecoordinatecurve}
\|\gamma(t)\|\:\leq C |t|,\qquad t\in[-T,T],
\end{equation}
holds.

We introduce one of the main objects of this paper, the {\em excess} of a
horizontal curve. Let us denote by 
$\opi: \mathfrak g\to\mathfrak g_1$ the
projection onto the first layer and  by $\pi : G\to\mathfrak g_1$   the map
$\pi := \opi \circ \exp ^{-1}$.
For any
curve $\gamma$ in $\G$ we use the short notation $\ugamma:=\pi\circ\gamma$. We  also identify $\g_1$ with $\R^r$ through the fixed orthonormal basis
$X_1,\dots,X_r$ and   denote by $S^{r-1}$ and $G(r-1)$
the set of unit vectors and linear hyperplanes in $\g_1$,
respectively. For the rest of this section, $I$
denotes a compact interval of positive length. 

	\begin{definition}\label{def:4.1}
	Given a  horizontal curve $\gamma:I\to\G$  and a Borel subset $B\subseteq I$ with $\leb^1(B)>0$,
		we define the \emph{excess} of $\gamma$ on $B$ as
		\[ \exc(\gamma;B):=\inf_{v\in
S^{r-1}}\pa{\media_B\ang{v,\underline{\dgamma}(t)}^2\,dt}^{1/2}. \]
			\end{definition}
	
	\begin{remark}\label{rem:excpiani}
		The excess can be equivalently defined as 
		\[ \exc(\gamma;B):=\inf_{\Pi\in
G(r-1)}\pa{\media_B\abs{\underline{\dgamma}(t)-\Pi\pa{\underline{\dgamma}(t)}}^2\,dt}^{1/2}, \]
		where we identify the hyperplane $\Pi$ with
the orthogonal
projection $\g_1\to\Pi$.  
	\end{remark}	
	
	\begin{remark}\label{rem:exctrdil}
		Given a horizontal curve $\gamma$, $g\in G$ and $r>0$, setting $\gamma_1(t):=g\,\gamma(t)$,
$\gamma_2(t):=\delta_r(\gamma(t))$, we have
		\[	
\exc(\gamma_1;B)=\exc(\gamma;B)\quad\text{and}\quad\exc(\gamma_2;B)=r\exc(\gamma
;B).
		\]
		Moreover, for $\gamma_3(t):=\delta_r(\gamma(t/r))$ we have
$\exc(\gamma_3;rB)=\exc(\gamma;B)$.
	\end{remark}
	
	\begin{remark}\label{exciscont}The map 
		\[ 
S^{r-1}\times
L^2(I,\mathfrak g_1) \ni (v,u) \mapsto\pa{\media_B\ang{v,u(t)}^2\,dt}^{1/2} \in\R
\]
		is continuous. As a consequence, the infimum in Definition \ref{def:4.1}
		is in fact a minimum and, by the compactness of $S^{r-1}$, we have
		\[ \exc(\gamma_k;B)\to\exc(\gamma;B) 
\]
		whenever $\underline{\dgamma_k}\to\underline{\dgamma}$ in $L^2(I,\mathfrak g_1)$.
	\end{remark}

	The following  compactness result for length minimizers parametrized by
arclength implies  
	a certain uniform -- though not explicit -- estimate: see Lemma
\ref{excrinc} below.
	
	\begin{lemma}[Compactness of minimizers]\label{cptmingrp} Let $I$ be a
compact interval and let $\gamma_k:I\to\G$, $k\in\N$,  be a
sequence of 
		length minimizers parametrized by arclength with
$\gamma_k(t_0)=0$, for a fixed $t_0\in I$. Then, there exist a subsequence $\gamma_{k_p}$ and a  length
minimizer
		$\gamma_\infty:I\to\G$, parametrized by arclength and with $\gamma_\infty(t_0)=0$, such that $\gamma_{k_p}\to\gamma_\infty$ uniformly
		and $\underline{\dgamma_{k_p}}\to\underline{\dgamma_\infty}$ in $L^2(I)$.
	\end{lemma}
	
	\begin{proof}
	By homogeneity, it is not restrictive to assume   $I=[0,1]$. 
		For any $k$ we have $\gamma_k([0,1])\subseteq\obar{B(0,1)}$, the
closed unit ball,  which is compact. Since all the curves
$\gamma_k$ are $1$-Lipschitz with respect to the Carnot--Carath\'eodory distance
$d$, we can find
		a subsequence $\gamma _{k_p}$ converging uniformly to some curve
$\gamma_\infty$.
		
	Let $u_p:=\underline{\dgamma_{k_p}}$. By \eqref{diffpi} one has $|u_p|=1$ a.e.,  so
up to selecting a further subsequence we can assume that
$u_p\rightharpoonup u_\infty$ in $L^2([0,1],\g_1)$. 
Thus, identifying $G$ with $\R^n$ by exponential coordinates and passing
to the limit as $p\to\infty$ in
		\[ 
\gamma_{k_p}(t)=\int_0^t\pa{\sum_{i=1}^r
u_{p,i}(\tau)X_i(\gamma_{k_p}(\tau))}\,d\tau
\]
		(which holds again by \eqref{diffpi}), we obtain, for any $t\in[0,1]$,
		\[ 
\gamma_\infty(t)=\int_0^t\pa{\sum_{i=1}^r u_{\infty,i}(\tau)X_i(\gamma_\infty(\tau))}\,d\tau. 
\]
		This proves that $\gamma_\infty$ is 
horizontal with $\underline{\dgamma_\infty}=u_\infty$. 		
Moreover,
		\begin{equation}
\label{eq:cptmineq}
		\norm{u_\infty}_{L^2([0,1],\g_1)}\ge L(\gamma_\infty)\ge d(\gamma_\infty(0),\gamma_\infty(1))=\lim_{p\to\infty}d(\gamma_{k_p}(0),\gamma_{k_p}(1))=1.
		\end{equation}
		We already know that $\norm{u_\infty}_{L^2([0,1],\g_1)}\le
1$ (because $u_p\rightharpoonup u_\infty$ and
$\|u_p\|_{L^2([0,1],\g_1)}=1$), so
		$\norm{u_p}_{L^2([0,1],\g_1)}
\to\norm{u_\infty}_{L^2([0,1],\g_1)}$ and, since $L^2([0,1],\g_1)$ is a
Hilbert space, this gives
		$u_p\to u_\infty$ in $L^2([0,1],\g_1)$.
In particular, $\underline{\dgamma_\infty}(t)$ is  for a.e.~$t\in[0,1]$  a unit vector in
$\g_1$. As all inequalities in \eqref{eq:cptmineq} must be equalities, we obtain
$L(\gamma_\infty)=d(\gamma_\infty(0),\gamma_\infty(1))$, i.e.,  $\gamma_\infty$
is a length minimizer parametrized by arclength.
	\end{proof}
	
	\begin{lemma}\label{excrinc} 
	For any  $\epsilon>0$   there exists a constant $c=c(G,\epsilon)>0$
such that the following holds. For any  length minimizer
$\gamma:I\to\G$   parametrized by arclength and such that
		$\exc(\gamma;I)\ge\epsilon$,  there exist $r$ subintervals
$[a_1,b_1],\dots,[a_r,b_r]\subseteq I$, with $a_i<b_i\le a_{i+1}$, such that
		\begin{equation}\label{eq:ildet}
		\abs{\det\pa{\ugamma(b_1)-\ugamma(a_1),\dots,\ugamma(b_r)-\ugamma(a_r)}}\ge c(\leb^1(I))^r.
		\end{equation}
The determinant is defined by means of the
identification of $\g_1$ with $\R^r$ via the basis
$X_1,\dots,X_r$.
	\end{lemma}
	
	\begin{proof}
	By Remark \ref{rem:exctrdil} we can  assume that $I=[0,1]$ 
		and that $\gamma(0)=0$.
		By contradiction, assume there exist 
length minimizers $\gamma_k:[0,1]\to\G$ parametrized by arclength,
		with $\gamma_k(0)=0$ and $\exc(\gamma_k;[0,1])\ge\epsilon$, 
such that
	
\begin{equation}\label{eq:absurdseq}\abs{\det\pa{\ubar{\gamma_k}(b_1)-\ubar{
\gamma_k}(a_1),\dots,\ubar{\gamma_k}(b_r)-\ubar{\gamma_k}(a_r)}}\le 2^{-k},
\end{equation}
		for any $0\le a_1<b_1\le\cdots\le a_r<b_r\le 1$.
		By Lemma \ref{cptmingrp}, there exists a subsequence
$(\gamma_{k_p})_{p\in\N}$ such that $\gamma_{k_p}\to\gamma_\infty$ uniformly
		and $\underline{\dgamma_{k_p}}\to\underline{\dgamma_\infty}$ in $L^2([0,1])$ for some length minimizer $\gamma_\infty$ parametrized by arclength.
		Passing to the limit  as $p\to\infty$ in \eqref{eq:absurdseq} we deduce that
		\begin{equation}\label{eq:detzero}
		\det\pa{\ubar{\gamma_\infty}(b_1)-\ubar{\gamma_\infty}(a_1),\dots,\ubar{\gamma_\infty}(b_r)-\ubar{\gamma_\infty}(a_r)}=0,
		\end{equation}
		for any $0\le a_1<b_1\le\cdots\le a_r<b_r\le 1$.
		
		Let  $S$ be the set  of   differentiability
points $t\in (0,1)$  of $\gamma_\infty$ and let
		\[ 
\mathfrak h_1:=\text{span}\{\underline{\dgamma_\infty}(t)\mid t\in S\}\subseteq \g_1 \]
		be the linear subspace of $\mathfrak g_1$ spanned by   the
derivatives $\underline{\dgamma_\infty}(t)$. We claim that $\dimens \mathfrak h_1<r$. If
this were not the case, we could find $0<t_1<\cdots<t_r<1, t_i\in S$, such that
$\underline{\dgamma_\infty}(t_1),\dots,
		\underline{\dgamma_\infty}(t_r)$ are linearly independent. Setting
		\[ a_i:=t_i,\ b_i:=t_i+\delta,\quad i=1,\dots,r \]
		and letting $\delta\downarrow 0$ in \eqref{eq:detzero}, 
we would deduce
that $\det\pa{\underline{\dgamma_\infty}(t_1),\dots,\underline{\dgamma_\infty}(t_r)}=0$, which is a
contradiction.
		
		As a consequence, there exists a unit vector $v \in\mathfrak
g _1$ orthogonal to $\mathfrak h_1$ and we   obtain
		\[ 
\exc(\gamma_\infty;[0,1])\le\pa{\int_0^1\ang{v,\underline{\dgamma_\infty}(t)}^2\,dt}^{1/2}
=0. 
\]
But from  $\exc(\gamma_{k_p};[0,1])\ge\epsilon$ and Remark \ref{exciscont} we
also
have $\exc(\gamma_\infty;[0,1])\ge\epsilon$. This is a contradiction and the
proof is accomplished.
\end{proof}

\begin{remark}\label{rem:elemdet}
Under the same assumptions and notation of Lemma \ref{excrinc},
we also have
\begin{equation}
\label{eq:bimbu}
 |\ugamma(b_i)-\ugamma(a_i) |\geq c\leb^1(I)\qquad \text{for any }i=1,\dots,r.
\end{equation}
Indeed, one has $ |\ugamma(b_i)-\ugamma(a_i) |\leq\leb^1(I)$ by arclength
parametrization and \eqref{eq:ildet} could not hold in case
\eqref{eq:bimbu} were false for some index $i$.
\end{remark}

     \section{Cut and correction devices}\label{sec:cutcor}

In this section we introduce the cut and the iterated correction    of a
horizontal curve. In Lemma \ref{excgain} we compute the gain of length in terms
of the excess. In the formula \eqref{BB}, we  establish an algebraic identity for
the displacement of the end-point produced by an iterated correction. We keep
on working in a Carnot group $G$.

The \emph{concatenation} of two   curves $\alpha:[a,a+a']\to\G$
        and $\beta:[b,b+b']\to\G$ is the curve
$\alpha*\beta:[a,a+(a'+b')]\to\G$ defined  by the
formula
        \[
        \alpha*\beta(t):=\begin{cases}\alpha(t) & \text{if }t\in[a,a+a'] \\
         \alpha(a+a')\beta(b)^{-1}\beta(t+b-(a+a')) & \text{if
}t\in[a+a',a+a'+b'].\end{cases}
          \]
The concatenation $\alpha*\beta$ is continuous if $\alpha$ and $\beta$ are continuous and
 it is horizontal
if $\alpha$ and $\beta$ are horizontal. The operation $*$ is
associative.

    \begin{definition}[Cut curve] \label{cutdef}
Let $\gamma:[a,b]\to\G$ be a   curve.
For any subinterval $[s,s']\subseteq[a,b]$ with $\ugamma(s')\neq \ugamma(s)$
we define the \emph{cut curve}
$\cut(\gamma;[s,s']): [a,b'']\to G$, with $b'' :=
b-(s'-s)+\abs{\ugamma(s')-\ugamma(s)}$, by the formula
        \[
\cut(\gamma;[s,s']):=\restr{\gamma}{[a,s]}*\restr{\exp(\:\cdot\:
w)}{\bra{0,\abs{\ugamma(s')-\ugamma(s)}}}*\restr{\gamma}{[s',b]},
 \] where
\[
 w := \frac{ \ugamma(s')-\ugamma(s)}{\abs{\ugamma(s')-\ugamma(s)}}.
\]
When $\ugamma(s')=\ugamma(s)$, the cut curve is defined by
        \[
\cut(\gamma;[s,s'])=\restr{\gamma}{[a,s]}*\restr{\gamma}{[s',b]}.
\]
    \end{definition}

        \begin{remark}\label{cutlen} If $\gamma$
is parametrized by arclength and horizontal, then the cut curve
$\cut(\gamma;[s,s'])$ is
  parametrized by arclength and horizontal, with length
        \begin{equation} \label{liar}
L(\cut(\gamma;[s,s']))=L(\gamma)-(s'-s)+|\ugamma(s')-\ugamma(s)| .
\end{equation}
    \end{remark}

    \begin{remark}\label{cutsameproj}
The final point of the cut curve has the same projection on $\g_1$ as the final
point of $\gamma$, i.e.,
        $\pi\pa{\cut(\gamma;[s,s'])(b'')}=\pi(\gamma(b)). $
Indeed, by Lemma \ref{pianyhomo} below we have
        \begin{dmath*}
        \pi\pa{\cut(\gamma;[s,s'])(b'')}
=\pi\big(\gamma(s)\exp\pa{\abs{\ugamma(s')-\ugamma(s)}w}\gamma(s')^{-1}
\gamma(b)\big)
=\ugamma(s)+\abs{\ugamma(s')-\ugamma(s)}w+\pa{\ugamma(b)-\ugamma(s')}
        =\ugamma(b). \end{dmath*}
    \end{remark}

    \begin{lemma}\label{excgain} Let
$\gamma:I\to\G$  be a  horizontal curve parametrized by arclength on a
compact interval  $I$  and let
$J\subseteq
I$ be a   subinterval with $\leb^1(J)>0$. Then we have
        \[ L(\gamma)-L(\cut(\gamma;J))\ge\frac{\leb^1(J)}{2}\exc(\gamma;J)^2. \]
    \end{lemma}

    \begin{proof} Let   $J=[s,s']$   for some $s<s'$.
        As in Definition \ref{cutdef}, let $w\in \g_1$ be  a unit
vector such that $\ang{w,\ugamma(s')-\ugamma(s)}=\abs{\ugamma(s')-\ugamma(s)}$,
i.e.,
        \[ \ang{w,\media_s^{s'}\underline\dgamma(t)\,dt}
        = \frac{|\ugamma(s')-\ugamma(s)|}{s'-s}. \]
        Since $\abs{\underline\dgamma}=1$ a.e., we have
        $
\abs{\underline\dgamma-w}^2=2\pa{1-\ang{w,\underline\dgamma}},
$
and since  $r\ge 2$ there exists a unit vector $v\in \g_1$ with
$\ang{v,w}=0$. Thus,
        for all $t$ such that $\underline\dgamma(t)$ is defined we have
        \[
\abs{\ang{v,\underline\dgamma(t)}}=\abs{\ang{v,\underline\dgamma(t)-w}}\le\abs{\underline\dgamma(t)-w}. \]
We deduce that
        \[
        \begin{split}
        \exc(\gamma;[s,s'])^2 & \le
\media_s^{s'}\ang{v,\underline\dgamma(t)}^2\,dt\le\media_s^{s'}\abs{\underline\dgamma(t)-w}^2\,dt\\
        & =  2\pa{1-\ang{w,\media_s^{s'}\underline\dgamma(t)\,dt}}
=2\pa{1-\frac{|\ugamma(s')-\ugamma(s)|}{s'-s}}.
        \end{split}
        \]
Multiplying by $\leb^1(J)=s'-s$ and using \eqref{liar}, we obtain the claim:
        \[ \leb^1(J)\exc(\gamma;J)^2\leq
2\pa{(s'-s)-\abs{\ugamma(s')-\ugamma(s)}}=2\pa{L(\gamma)-L(\cut(\gamma;J))}. \qedhere
\]

    \end{proof}

Given $Y\in\g$, we hereafter denote
by $\delta_Y:[0,\ell_Y]\to\G$ any geodesic from $0$ to $\exp(Y)$ parametrized by
arclength (the choice of $\delta_Y$ is not important); in particular, $\ell_Y=d(0,\exp(Y))$.
We denote by  $\delta_Y(\ell_Y-\cdot\:)$   the curve $\delta_Y$ traveled
backwards from $\exp(Y)$ to $0$.

\newcommand{\Dis}{\mathrm{Dis}}

    \begin{definition} [Corrected curve and displacement]
\label{def:cor}
Let $\gamma:[a,b]\to\G$ be a  horizontal curve  pa\-ra\-me\-trized
by arclength. For any   subinterval $[s,s']\subseteq[a,b]$  and
$Y\in\g$,
        we define the \emph{corrected curve} $\dev(\gamma;[s,s'],Y):[a,b''']\to\G$, with
$b''':=b+2\ell_Y$,  by
        \[
\dev(\gamma;[s,s'],Y):=\restr{\gamma}{[a,s]}*\delta_Y*\restr{\gamma}{[s,s']}
*\delta_Y(\ell_Y-\cdot\:)*\restr{\gamma}{[s',b]}.
\]
We   refer to the process of transforming $\gamma$ into
$\dev(\gamma;[s,s'],Y)$ as to the application of the \emph{correction device}
        associated with $[s,s']$ and $Y$.    The {\em displacement}
of the final point produced by the
correction device associated with $[s,s']$ and $Y$ is
\[
 \Dis(\gamma;[s,s'],Y):=
\gamma(b)^{-1}\dev(\gamma;[s,s'],Y)(b''').
\]
    \end{definition}

We will later express the displacement in terms of suitable conjugations 
$C_g(h):=ghg^{-1}$   and  commutators $[g,h]:=ghg^{-1}h^{-1}$ in $\G$. 

For any $1\le j\le s$, we denote by $\opi_j:\g\to \g_j$ the
canonical projection with respect to the direct sum. 
The mappings  $\pi_j: G\to \g$ are defined as
 $\pi_j := \opi_j \circ \exp^{-1}$. Clearly, one has $\opi_1=\opi$ and $\pi_1=\pi$.
We let $\mathfrak w_j := \g_j\oplus\cdots\oplus\g_s$
and  $G_ j :=\exp (\mathfrak w_j)$. We also agree that   $G_{s+1}:=\{0\}$, the
identity element of $G$, and $\mathfrak
w_{s+1}:=\{0\}$.

	\begin{lemma}\label{pianyhomo} The map
$\pi: G\to (\g_1,+)$ is a group homomorphism and
\begin{equation}\label{diffpi} \pi_*X_i=X_i\qquad \text{for }i=1,\dots,r. \end{equation}
For any $1\le j\le s$,
$G_j$
is a  subgroup of $G$ and  
$\pi_j :G_j \to(\g_j,+)$ is a group
homomorphism.
	\end{lemma}

	\begin{proof}
Given points
$g=\exp(x_1X_1+\dots+x_nX_n)$
		and $g'=\exp(x_1'X_1+\dots+x_n'X_n)$ in $G$, by  
\eqref{eq:bchpol} we have 
		$\exp^{-1}(g\, g')= (x_1+x_1')X_1+\dots+(x_r+x_r')X_r+R$ 
with $R\in \mathfrak w  _2 $ and hence
		\[
\pi(g g')=\obar{\pi}(\exp^{-1}
(g g'))=(x_1+x_1')X_1+\dots+(x_r+x_r')X_r=\pi(g)+\pi(g').
\]
The identity \eqref{diffpi} follows from this and the left-invariance of $X_i$.
The fact that $G_j$ is a subgroup follows from the   Baker--Campbell--Hausdorff 
formula, and the assertion that $\pi_j: G_j\to \g_j$ is a homomorphism can be
obtained as above.
	\end{proof}

The following lemmas describe how  the
homomorphisms $\pi_j$ interact with conjugations, commutators and Lie brackets.
We denote by  $\Ad(g)$ the differential of the conjugation $C_g$ at the
identity  $0 \in G$. This
is an
automorphism of
$T_0G=\g$.
For  $X,Y\in \g$ and $g\in G$,
		we have the  formulas
$\Ad(\exp(X))=\mathrm{e}^{\ad(X)} $ and $C_g(\exp(Y))=\exp(\Ad(g)Y)$, see e.g.~\cite[Proposition~1.91]{knapp}.

	\begin{lemma}\label{conjsameproj} 
For any   $g\in  G$ and $h\in  G_j$ we have  $ghg^{-1}\in
G_j$ (i.e., $G_j$ is normal in $G$) and    
$\pi_j(ghg^{-1})=\pi_j(h)$.
	\end{lemma}

	\begin{proof} With $g=\exp(X)$ and $h =\exp(Y)$, we have 
		\[ 
\exp^{-1}(ghg^{-1})=\Ad(g)Y=\mathrm{e}^{\ad
X}Y=\sum_{k=0}^\infty\frac{(\ad X)^k}{k!}Y=Y+R, \]
with  $R\in \mathfrak w _{j+1}$, because in the previous sum all
the terms with
$k\ge 1$ belong
to $\mathfrak w_{j+1}$.
		Hence, we have $ghg^{-1}\in G _j$ and
		\[
\pi_j(ghg^{-1})=\opi_j\circ\exp^{-1}(ghg^{-1})=\opi_j(Y+R)=\opi_j(Y)=\pi_j(h). \qedhere
\]
\end{proof}

	\begin{lemma}\label{commproj} For any
$g\in G$ and
$h\in G_j$ with  $1\le j<s$ we have 
		\[ 
[g,h]\in G_{j+1} \quad\text{and}\quad 
\pi_{j+1}([g,h])=[\pi(g),\pi_j(h)]. \]
A similar statement holds if $g\in G_j$ and $h\in G$.
	\end{lemma}

	\begin{proof}
We prove only the first part of the statement,
the second one following
from the first one and the identity $[g,h]=[h,g]^{-1}$.
	Combining Lemma \ref{conjsameproj} with Lemma
\ref{pianyhomo}, we obtain
		$[g,h]=(ghg^{-1})h^{-1}\in G_j$ and
		\[
\pi_j([g,h])=\pi_j(ghg^{-1})+\pi_j(h^{-1})=\pi_j(h)-\pi_j(h)=0, \]
		so that $[g,h]\in G_{j+1}$. Now, writing $g=\exp(X)$,
$h=\exp(Y)$ and using the formula $\exp^{-1}(ghg^{-1})=\mathrm{e}^{\ad X}Y$
		as in the previous proof, we obtain
		\[ \exp^{-1}(ghg^{-1})=\sum_{k=0}^\infty\frac{(\ad
X)^k}{k!}Y=Y+[X,Y]+R', \]
		where the remainder $R'$ is the sum of all terms with $k\ge 2$
and thus belongs to $\mathfrak w_{j+2}$.
		As $h^{-1}=\exp(-Y)$, the Baker--Campbell--Hausdorff  
formula gives
		\[
 \exp^{-1}([g,h])=P(Y+[X,Y]+R',-Y)=[X,Y]+R'+R'', \]
		where $R''$ is given by the commutators of length at least $2$ appearing in \eqref{eq:bchpol}.
Now, thinking each such commutator
		as a $(k_1+\ell_1+\dots+k_p+\ell_p)$-multilinear function (and
expanding each instance of $Y+[X,Y]+R'$ accordingly),
		we obtain that $R''$ is a linear combination of elements of the
form
		\[ (\ad Z_1)\cdots(\ad Z_k)Z_{k+1}, \]
		where $k\ge 1$ and $Z_i\in\{Y,[X,Y],R'\}$. Those elements
where only $Y$ appears vanish,
		while the other terms belong to $\mathfrak w_{j+2}$, since
$[X,Y],R'\in
\mathfrak w_{j+1}$ and $k\ge 1$. We deduce that $R''\in \mathfrak w_{j+2}$.
		Finally,
		\[
\pi_{j+1}([g,h])=\opi_{j+1}([X,Y]+R'+R'')=\opi_{j+1
}([X,Y])
		=[\opi(X),\opi_j(Y)], \]
		since $X=\opi(X)+R_X$ and $Y=\opi_j(Y)+R_Y$, with $R_X\in
\mathfrak w_2$
and $R_Y\in \mathfrak w_{j+1}$.
%
	\end{proof}

Hereafter, we  adopt the short notation
$\incr{\gamma}_a^b:=\gamma(a)^{-1}\gamma(b)$.

    \begin{lemma}\label{lem:corstratoj}
    Under the   assumptions and notation of Definition \ref{def:cor}, the
displacement  is given by the formula
        \begin{equation}
\label{eq:deveffect}
 \Dis(\gamma;[s,s'],Y) =C_{\incr{\gamma}_b^s}
        \pa{\bra{\exp(Y),\incr{\gamma}_s^{s'}}}. \end{equation}
    In particular, if $Y\in \g_j$ and $1\leq j<s$, then
$
 \Dis(\gamma;[s,s'],Y) \in\G_{j+1}$ and
        \[ \pi_{j+1}\pa{
 \Dis(\gamma;[s,s'],Y) }=[Y,\ugamma(s')-\ugamma(s)]. \]
    \end{lemma}

    \begin{proof}
    We have
        \begin{dmath*}
\dev(\gamma;[s,s'],Y)(b''')=\gamma(s)\exp(Y)\incr{\gamma}_s^{s'}\exp(-Y)\incr{
\gamma}_{s'}^b           
=\gamma(s)\bra{\exp(Y),\incr{\gamma}_s^{s'}}\incr{\gamma}_s^{s'}\incr{\gamma}_{
s'}^b
            =\gamma(s)\bra{\exp(Y),\incr{\gamma}_s^{s'}}\incr{\gamma}_s^b,
\end{dmath*}
     hence
        \[
        \Dis(\gamma;[s,s'],Y)
)=\incr{\gamma}_b^s\bra{\exp(Y),\incr{\gamma}_s^{s'}}\pa{\incr{\gamma}_b^s}^{-1}
    =C_{\incr{\gamma}_b^s}\pa{\bra{\exp(Y),\incr{\gamma}_s^{s'}}}.
             \]
By Lemma \ref{pianyhomo}, we have $
\pi(\incr{\gamma}_s^{s'})=\ugamma(s')-\ugamma(s)$; moreover, $\pi_j(\exp(Y))=Y$.
Hence, using Lemma \ref{commproj}, we obtain
        \[ \bra{\exp(Y),\incr{\gamma}_s^{s'}}\in
G_{j+1}\quad\text{and}\quad\pi_{j+1}\pa{\bra{\exp(Y),\incr{\gamma}_s^{s'}}}=[Y,
\ugamma(s')-\ugamma(s)]. \]
        The lemma now follows from equation \eqref{eq:deveffect} and Lemma
\ref{conjsameproj}.
    \end{proof}

    \begin{definition}[Iterated correction]\label{coupledev}
Let   $\gamma:I\to\G$ be a horizontal curve
pa\-ra\-me\-trized by arclength on the interval $I$ and let
$I_1:=[s_1,t_1],\dots,I_k:=[s_k,t_k]\subseteq I$ be subintervals
        with $t_i\le s_{i+1}$. For any
$Y_1,\dots,Y_k\in\g$ we define by induction on $k\geq 2$ the iterated
correction
        \begin{equation*}
          \dev(\gamma;I_1,Y_1;\dots;I_k,Y_k)\\
        :=   \dev(
\dev(\gamma;I_1,Y_1;\dots;I_{k-1},Y_{k-1});I_k+2{\textstyle\sum_{i<k}}\ell_{Y_i},Y_k).
        \end{equation*}
    \end{definition}

The iterated correction is a curve defined on the interval $[a,\widehat b]$,
with $\widehat b := b+2 \sum_{i=1}^k \ell_{Y_{i}} $.
The displacement
of the final point produced by this iterated correction  is
\[
 \Dis
(\gamma;I_1,Y_1;\dots;I_k,Y_k) :=
\gamma(b)^{-1}
\dev
(\gamma;I_1,Y_1;\dots;I_k,Y_k) (\widehat b).
\]

    \begin{corollary}\label{mdeveffectuse}
For any  $I_i=[s_i,t_i]\subseteq I$ and  $Y_i\in \g_j$,  with $i=1,\ldots,k$ and
$j<s$,
 we have
        \begin{equation}\label{AA}
 \Dis
(\gamma;I_1,Y_1;\dots;I_k,Y_k) \in\G_{j+1} \end{equation}
        and
        \begin{equation}\label{BB}
         \pi_{j+1}\pa{
 \Dis
(\gamma;I_1,Y_1;\dots;I_k,Y_k)}=\sum_{i=1}^k[Y_i,\ugamma(t_i)-\ugamma(s_i)].
\end{equation}
    \end{corollary}

    \begin{proof} We prove \eqref{BB} by induction on $k$. The case $k=1$ is in
Lemma \ref{lem:corstratoj}.
    Assume the formula holds for $k-1$. Letting $\widehat{\gamma} :=
\dev(\gamma;I_1,Y_1;\ldots;I_{k-1}, Y_{k-1})$, which is defined on the interval
$[a,\widehat b]$ (where $\widehat{b}:=b+2\sum_{i<k}\ell_{Y_i}$),
    we have
    \[
    \begin{split}
      \Dis(\gamma;I_1,Y_1; \ldots; I_k,Y_k) & = \gamma(b)^{-1}\dev(\widehat{
\gamma}; I_k+(\widehat{b}-b),Y_k)
      \\
      & = \gamma(b)^{-1}~\widehat{\gamma} (\widehat{b})~\Dis(\widehat{ \gamma};
I_k+(\widehat{b}-b),Y_k)
      \\
      & = \Dis (\gamma;I_1,Y_1;\ldots; I_{k-1},Y_{k-1})~\Dis(\widehat {\gamma};
I_k+(\widehat{b}-b),Y_k).
    \end{split}
    \]
    Then, by Lemma \ref{pianyhomo}, by the inductive assumption and by Lemma
\ref{lem:corstratoj} applied to $\widehat{\gamma}$
we
have
    \[
    \begin{split}
     & \pi_{j+1}\big(  \Dis(\gamma;I_1,Y_1; \ldots; I_k,Y_k) \big)\\
      = &\sum_{i=1}^{k-1} [ Y_i,\ugamma(t_i)-\ugamma(s_i)]
        + [ Y_k, \widehat{\ugamma}(t_k+(\widehat{b}-b))-\widehat{\ugamma}(s_k+(\widehat{b}-b))]
\\
= & 
\sum_{i=1}^{k} [ Y_i,\ugamma(t_i)-\ugamma(s_i)],
    \end{split}
    \]
    because $\widehat{\ugamma}(t_k+(\widehat{b}-b))-\widehat{\ugamma}(s_k+(\widehat{b}-b))=
     \ugamma (t_k)- \ugamma (s_k)$.
    \end{proof}

When dealing with curves $\gamma$ defined on symmetric
intervals, it is convenient to use
        modified versions of $\cut$ and $\dev$, which we will denote by
$\cut'(\gamma;[s,s'])$ and $\dev'(\gamma;[s,s'],Y)$. They  are obtained
from $\cut(\gamma;[s,s'])$ and $\dev(\gamma;[s,s'],Y)$ by composition with the
time translation   such  that the new domain is  a symmetric
interval. The iterated correction is then defined in the following way:
        \begin{equation*}
          \dev'(\gamma;I_1,Y_1,\dots;I_k,Y_k)\\
        :=   \dev'(
\dev'(\gamma;I_1,Y_1;\dots;I_{k-1},Y_{k-1});I_k+{\textstyle\sum_{i<k}}\ell_{Y_i},Y_k).
        \end{equation*}
The related displacement  satisfies the properties \eqref{AA} and \eqref{BB} of
Corollary
\ref{mdeveffectuse} with $\dev'$ replacing $\dev$.

    \section{Proof of the main results} \label{sec:proofs}

\renewcommand{\rho}{\varrho}

Let $G$ be a Carnot group with rank $r\geq 2$ and step $s$, and let
$\mathcal X = \{X_1,\ldots,X_r\}$ be an orthonormal basis for $\mathfrak g_1$ (recall that $\g$ is endowed with a scalar product such that $\g_i\perp\g_j$). We first prove the one-sided version of Theorem \ref{1.2}; we 
will
illustrate later how to adapt the   proof     in order to obtain
Theorem \ref{1.2}.

    \begin{theorem}\label{onesthm}
    Let $\gamma:[0,T]\to G$, $T>0$, be a length-minimizing curve
parametrized by arclength.     Then  there exists an infinitesimal sequence
$\eta_i\downarrow 0$ such that
    \begin{equation} \label{ecconesided}
        \lim_{i\to\infty} \Exc(\gamma;[0,\eta_i]) = 0.
    \end{equation}
    \end{theorem}

    \begin{proof}
        \emph{Step 1.} We can assume that $\gamma(0)=0$.
Suppose by contradiction that there exists $\epsilon>0$ such that
$\exc(\gamma;[0,t])\ge\epsilon$ for any sufficiently small $t>0$.
        For $k=1,\dots,s$, we   inductively
define  horizontal curves $\gamma^{(k)}:[0,T_k]\to\G$  parametrized by
arclength such that:
        \begin{enumerate}[label=(\roman*)]
            \item $\gamma^{(k)}(0)=\gamma(0)=0$;
            \item $\gamma(T)^{-1}\gamma^{(k)}(T_k)\in\G_{k+1}$,
where    $\G_{s+1}=\set{0}$;
            \item $L\pa{\gamma^{(k)}}<L(\gamma)$, i.e., $T_k<T$.
        \end{enumerate}
        In particular, $\gamma^{(s)}$ is a horizontal curve with the
same endpoints as $\gamma$, but with smaller length:
        this   contradicts the minimality of $\gamma$.

        We define $\gamma^{(1)}:=\cut(\gamma;[0,\eta])$, where the
parameter $\eta>0$ will be chosen later; in fact, any sufficiently small $\eta$
will work. In this proof, the notation $O(\cdot)$ and $o(\cdot)$
        is used for asymptotic estimates which hold as $\eta\to 0$.
        By Remark \ref{cutsameproj}
and Lemma \ref{excgain}, $\gamma^{(1)}$ satisfies (i), (ii) and (iii)
with $k=1$.
        \medskip

        \emph{Step 2.} 
Let us fix parameters $\beta>0$ and
$\rho_1:=1>\rho_2>\dots>\rho_s>0$ such that
        for all $k=1,\dots,s-1$
        \begin{equation}\label{bel}
\frac{(k+1)\rho_k-\rho_{k+1}}{k}>1+\duebeta .
\end{equation} This is possible if $\beta$ is small
enough: indeed, the   inequality \eqref{bel} is equivalent to
        \[ \rho_k>\frac{\rho_{k+1}+k}{k+1}+\frac{k}{k+1}\beta,
 \]
        and we can choose any $\rho_s\in (0,1)$ and
then $\rho_{s-1}<1$ so as to verify the (strict) inequality
        when $\beta=0$ and $k=s-1$, then $\rho_{s-2}$ similarly and so on, up to $\rho_1=1$.
        By continuity, the inequalities will still hold for a small enough
$\beta>0$.

        For any $k=1,\dots,s$, we set $I_k:=[0,\eta^{\rho_k}]$;
the curve $\gamma^{(k+1)}$ is defined from $\gamma^{(k)}$ by applying
several correction devices within $I_{k+1}$, see \eqref{eq:defgammakpiuuno}.
As soon as $\eta\le 1$, the inclusions 
        $ [0,\eta]=I_1\subseteq I_2\subseteq\cdots\subseteq I_{s}$ hold.

        By Lemma \ref{excgain}, since $\exc(\gamma;[0,\eta])\ge\epsilon$,
the  gain of length obtained by performing the cut is
        \[
L(\gamma)-L(\gamma^{(1)})\ge\frac{\eta\epsilon^2}{2}\ge\eta^{1+\duebeta}, \]
        provided $\eta$ is small enough.

        The curves $\gamma^{(k)}:[0,T_k]\to\G$ will be constructed inductively
so as to satisfy (i), (ii) and (iii),
        as well as the following additional technical properties, which
 hold for $\gamma^{(1)}$:
        \begin{enumerate}[label=(\roman*)]
            \setcounter{enumi}{3}
            \item $T_k\ge T_{k-1}$ if $k\ge 2$;
            \item $L\pa{\gamma^{(k)}}\le
L(\gamma)-(1+o(1))\eta^{1+\duebeta}$;
            \item ${\ubar{\gamma^{(k)}}}(t)={\ugamma}(t+(T-T_k))$ for any
$t\in[2\eta^{\rho_k},T_k]$, i.e., on $[2\eta^{\rho_k},T_k]$ the curve
$\gamma^{(k)}$ has the same projection on $\g_1$  as the corresponding final
piece of $\gamma$;
            \item
$\norm{\ubar{\gamma^{(k)}}-\restr{\ugamma}{[0,T_k]}}_{L^\infty}=O(\eta)$.
        \end{enumerate}
        Notice that (v) implies (iii) for small enough $\eta$.

\medskip

        \emph{Step 3.} Assume that $\gamma^{(k)}$ has been constructed,
for some $1\le k\le s-1$, in such a way that (i)--(vii) hold. By (ii), there
exists
 $E_k\in \g_{k+1}\oplus\cdots\oplus \g_s$ such that
        \[ \gamma(T)^{-1}\gamma^{(k)}(T_k)=\exp(E_k) .
\]
Let us estimate $\opi_{k+1}(E_k)$. First, by (vi) and the
uniqueness of horizontal lifts, we have
        \[
\gamma^{(k)}\big|_{2\eta^{\rho_k}}^{T_k}=\gamma\big|_{\tau_k}^T,\qquad\text{
where $\tau_k:=2\eta^{\rho_k}+(T-T_k)$.} \]
         Hence, defining $g_k:=\gamma(\tau_k)^{-1}\gamma^{(k)}(2\eta^{\rho_k})$,
we have
        \[
        \begin{split}
         \gamma^{(k)}(T_k)= &
\gamma^{(k)}(2\eta^{\rho_k})\incr{\gamma^{(k)}}_{2\eta^{\rho_k}}^{T_k}
            =\gamma(\tau_k)g_k\incr{\gamma}_{\tau_k}^T\\
            = &\gamma(\tau_k)\incr{\gamma}_{\tau_k}^T
            C_{\incr{\gamma}_T^{\tau_k}}(g_k)
            =\gamma(T)C_{\incr{\gamma}_T^{\tau_k}}(g_k),
        \end{split}
        \]
        i.e.,
$g_k=C_{\incr{\gamma}_{\tau_k}^T}(\gamma(T)^{-1}\gamma^{(k)}(T_k))$. From (ii) and Lemma \ref{conjsameproj} we obtain
        $g_k\in\G_{k+1}$ and
\begin{equation}\label{eq:CDG}
\opi_{k+1}(E_k)=\pi_{k+1}\pa{\gamma(T)^{-1}\gamma^{(k)}(T_k)}=\pi_{k+1}
(g_k)=O\pa{\eta^{(k+1)\rho_k}}.
\end{equation}
        The last estimate follows from \eqref{rem:potenzecoordinatecurve} applied to the curve
        \[
        \gamma(\tau_k)^{-1}\
\restr{\gamma}{[0,\tau_k]}(\tau_k-\cdot)*\restr{\gamma^{(k)}}{[0,2\eta^{\rho_k}]
},
        \]
        which connects $0$  to
$\gamma(\tau_k)^{-1}\gamma^{(k)}(2\eta^{\rho_k})$. Its length
is $\tau_k+2\eta^{\rho_k}$ and is controlled by $ 5\eta^{\rho_k}$ because, by
(iv),
        \[
        T-T_k\le T-T_1 = L(\gamma)-L(\gamma^{(1)})\le\eta\le\eta^{\rho_k}.
        \]

\medskip

        \emph{Step 4.} We now define $\gamma^{(k+1)}$. As
$\g_{k+1}=[\g_k,\g_1]$, using   estimate \eqref{eq:CDG} for $\opi_{k+1}(E_k)$, there exist $Y_1,\dots,Y_r\in \g_k$ such that
\begin{equation}
\label{LUM}
\opi_{k+1}(E_k) = \sum_{i=1}^r[Y_i,X_i]\quad\text{and}\quad
\abs{Y_1},\dots,\abs{
Y_r } =O\pa { \eta^ { (k+1)\rho_k}}.
\end{equation}
        Furthermore, we have $ \exc(\gamma;I_{k+1})\ge\epsilon$ whenever
$\eta$ is small enough.
We can then apply Lemma \ref{excrinc} to
$I_{k+1}$ and find $[a_1,b_1],\dots,[a_r,b_r]\subseteq I_{k+1}$
        (with $b_i\le a_{i+1}$) such that
        \[
\abs{\det\pa{\ugamma(b_1)-\ugamma(a_1),\dots,\ugamma(b_r)-\ugamma(a_r)}}\ge
c\eta^{r\rho_{k+1}}. \]
        Using (vii) we obtain, for small
$\eta$,
        \[
        \begin{split}
\abs{\det\pa{\ubar{\gamma^{(k)}}(b_1)-\ubar{\gamma^{(k)}}(a_1),\dots,\ubar{
\gamma^{(k)}}(b_r)-\ubar{\gamma^{(k)}}(a_r)}}
&\ge c\eta^{r\rho_{k+1}}
        -O(\eta^{1+(r-1)\rho_{k+1}})\\
        &\ge\frac{c}{2}\eta^{r\rho_{k+1}}. 
        \end{split}
        \]
        This   implies that for $i=1,\dots,r$ we
have
        \[ X_i=\sum_{j=1}^r
c_{ij}\pa{\ubar{\gamma^{(k)}}(b_j)-\ubar{\gamma^{(k)}}(a_j)}, \]
        where $\abs{c_{ij}}=O\pa{\eta^{-\rho_{k+1}}}$. This estimate
depends on $c$ and thus on $\epsilon$.
        So, defining $Z_j=\sum_{i=1}^r c_{ij}Y_i$, from \eqref{LUM} we
obtain
        \[ \opi_{k+1}(E_k)=\sum_{j=1}^r
[Z_j,\ubar{\gamma^{(k)}}(b_j)-\ubar{\gamma^{(k)}}(a_j)], \]
        with $\abs{Z_j}=O\pa{\eta^{(k+1)\rho_k-\rho_{k+1}}}$.
        Finally, we let
        \begin{equation}\label{eq:defgammakpiuuno}
        \gamma^{(k+1)}:=\dev(\gamma^{(k)};[a_1,b_1],-Z_1;\dots;[a_r,b_r],-Z_r).
        \end{equation}
        Since $d(0,\exp(Z))=O(|Z|^{1/k})$ for $Z\in\g_k$,
the extra length $T_{k+1}-T_k$ needed for the application of  these $r$
correction devices is
        \[
T_{k+1}-T_k=\sum_{j=1}^r
O\pa{\abs{Z_j}^{1/k}}=O\pa{\eta^{\frac{(k+1)\rho_k-\rho_{k+1}}{k}}}=
o\pa{\eta^{1+\duebeta}},
\]
        thanks to the inequalities \eqref{bel}   on the parameters
$\rho_k$. Thus, we obtain
        \[
L(\gamma^{(k+1)})\le L(\gamma^{(k)})+o(\eta^{1+\duebeta}).
\]

\medskip

        \emph{Step 5.} We check that $\gamma^{(k+1)}$  satisfies
properties (i)--(vii). We have just verified
        (iii) and (v), while (i) and (iv)  are trivial. The property
(vii) follows from the fact that $\gamma^{(k+1)}$ (as well as
$\ubar{\gamma^{(k+1)}}$) is obtained from $\gamma^{(k)}$ (from
$\ubar{\gamma^{(k)}}$) by the application of correction devices of total length
$o(\eta^{1+\duebeta})=O(\eta)$.

        In order to check (vi), we remark that
        \[
\ubar{\gamma^{(k+1)}}=\restr{\ubar{\gamma^{(k+1)}}}{[0,\eta^{\rho_{k+1}}+(T_{k+1
}-T_k)]}*\restr{\ubar{\gamma^{(k)}}}{[\eta^{\rho_{k+1}},T_k]} \]
        and that the final point of the first curve in this concatenation
coincides with the starting point of the second one.
        Since $   
T_{k+1}-T_k=O\pa{\eta^{\frac{(k+1)\rho_k-\rho_{k+1}}{k}}}=o\pa{\eta^{\rho_{k+1}}
}$,
        if $\eta$ is small enough we obtain
        \begin{dmath*}
\restr{\ubar{\gamma^{(k+1)}}}{[2\eta^{\rho_{k+1}},T_{k+1}]}
=\restr{\ubar{\gamma^{(k)}}}{[2\eta^{\rho_{k+1}}-(T_{k+1}-T_k),T_k]}\pa{\:\cdot\,-(T_{k+1}-T_k)}
=\restr{\ugamma}{[2\eta^{\rho_{k+1}}+(T-T_{k+1}),T]}\pa{\:\cdot\,+(T-T_{k+1})}
,\end{dmath*}
        the last equality holding by hypothesis (vi) for $k$, because $2\eta^{\rho_{k+1}}-(T_{k+1}-T_k)\ge
2\eta^{\rho_k}$ when $\eta$ is small.
        Thus, $\gamma^{(k+1)}$ satisfies (vi).

        Finally, let us check (ii).
By  Lemma \ref{pianyhomo} and
Corollary \ref{mdeveffectuse}, we have
        \[
\gamma(T)^{-1}\gamma^{(k+1)}(T_{k+1})=\pa{\gamma(T)^{-1}\gamma^{(k)}(T_k)}\pa{
\gamma^{(k)}(T_k)^{-1}\gamma^{(k+1)}(T_{k+1})}\in\G_{k+1} \]
        and
        \[ \begin{split}\pi_{k+1}\pa{\gamma(T)^{-1}\gamma^{(k+1)}(T_{k+1})}     
 &=\pi_{k+1}(\exp(E_k))+\pi_{k+1}\pa{\gamma^{(k)}(T_k)^{-1}\gamma^{(k+1)}(T_{k+1
} )} \\
        &=\opi_{k+1}(E_k)+\sum_{i=1}^r[-Z_i,\ugamma(b_i)-\ugamma(a_i)] \\
        &=0. \end{split} \]
        This concludes the proof.
    \end{proof}

We now prove Theorem \ref{1.2}. The proof is basically the same as that of
Theorem \ref{onesthm} and we just list the required minor modifications below.

    \begin{proof}[Proof of Theorem \ref{1.2}] The constraints imposed on the
curves $\gamma^{(k)}$,
        as well as the cut and correction operations, have to be
replaced by their symmetric counterparts.
        For $k=1,\dots,s$ we   inductively construct horizontal curves
$\gamma^{(k)}:[-T_k,T_k]\to\G$  parametrized by arclength satisfying:
        \begin{enumerate}[label=(\roman*')]
            \item $\gamma^{(k)}(-T_k)=\gamma(-T)$;
            \item $\gamma(T)^{-1}\gamma^{(k)}(T_k)\in\G_{k+1}$ (in particular,
$\gamma^{(s)}(T_s)=\gamma(T)$);
            \item $L\pa{\gamma^{(k)}}<L(\gamma)$, i.e., $T_k<T$;
            \item $T_k\ge T_{k-1}$ if $k\ge 2$;
            \item $L\pa{\gamma^{(k)}}\le L(\gamma)-(1+o(1))\eta^{1+\duebeta}$;
            \item
$\restr{\ubar{\gamma^{(k)}}}{[2\eta^{\rho_k},T_k]}=\restr{\ugamma}{[2\eta^{
\rho_k}+(T-T_k),T]}(\:\cdot\,+(T-T_k))$ and
\\
$\restr{\ubar{\gamma^{(k)}}}{[-T_k,-2\eta^{\rho_k}]}=\restr{\ugamma}{[-T,-2\eta^
{\rho_k}-(T-T_k)]}(\:\cdot\,-(T-T_k))$;
            \item
$\norm{\ubar{\gamma^{(k)}}-\restr{\ugamma}{[-T_k,T_k]}}_\infty=O(\eta)$.
        \end{enumerate}
        We   list the necessary modifications in the various steps.

        {\em Step 1.} The first curve is
$\gamma^{(1)}:=\cut'(\gamma;[-\eta,\eta])$, which satisfies (i')--(vii') for
$k=1$.

        {\em Step 2.} The interval $I_k$ is now
$[-\eta^{\rho_k},\eta^{\rho_k}]$.

        {\em Step 3.}~Let $E_k,\tau_k,g_k$ be defined as in the proof of Theorem
\ref{onesthm}. The estimate $\pi_{k+1}(g_k)=O\pa{\eta^{(k+1)\rho_k}}$ follows by
applying \eqref{rem:potenzecoordinatecurve} to the curve
        \begin{equation}\label{eq:curvasopra}
        \gamma(\tau_k)^{-1}\
\restr{\gamma}{[-\tau_k,\tau_k]}(\tau_k-\cdot)*\restr{\gamma^{(k)}}{[-2\eta^{
\rho_k},2\eta^{\rho_k}]}
        \end{equation}
        and observing that  $\gamma(-\tau_k)
=\gamma^{(k)}(-2\eta^{\rho_k})$. The length of
the curve in \eqref{eq:curvasopra} is $2\tau_k+4\eta^{\rho_k}\leq
10\eta^{\rho_k}$.

        {\em Step 4}. In the definition \eqref{eq:defgammakpiuuno}
of $\gamma^{(k+1)}$, $\dev$ is replaced by $\dev'$.

        {\em Step 5}. The fact that $\gamma^{(k+1)}$  satisfies (vi')
  follows from the identity \[
\ubar{\gamma^{(k+1)}}=\restr{\ubar{\gamma^{(k)}}}{[-T_k,-\eta^{\rho_{k+1}}]}
*\restr{\ubar{\gamma^{(k+1)}}}{J_k}*\restr{\ubar{\gamma^{(k)}}}{[\eta^{\rho_{k+1}},T_k]}(\cdot+T_{k+1}-T_k), \]
with $J_K:=[-\eta^{\rho_{k+1}}-(T_{k+1}-T_k),\eta^{\rho_{k+1
}}+(T_{k+1}-T_k)]$,        where the final point of each curve in the concatenation coincides with
the starting point of the next one.
    \end{proof}

We finally prove Theorem \ref{1.1} and then state its one-sided version.

\begin{proof}[Proof of Theorem \ref{1.1}]
As mentioned in the introduction, it is not restrictive to assume
that the Carnot--Carathéodory structure $(M,\mathcal X)$ is  that of a Carnot
group $G$. 
To see this, consider the following facts: 
\begin{itemize}
\item[(i)] if $\gamma $ is length-minimizing in $ (M,\mathcal X)$ and $\kappa \in
\Tan(\gamma;t)$, then $\kappa$ is length-minimizing in the nilpotent approximation $(M^\infty,\mathcal X^\infty)$, see \cite[Theorem~3.6]{MPVconotangente};
\item[(ii)] if $\kappa \in\Tan(\gamma;t)$ and $\widehat \kappa \in
\Tan(\kappa;0)$, then $\widehat \kappa \in\Tan(\gamma;t)$, see \cite[Proposition~3.7]{MPVconotangente}.
\end{itemize}
As a consequence, it suffices to show Theorem \ref{1.1} for $(M^\infty,\mathcal{X}^\infty)$. Finally:
\begin{itemize}
\item[(iii)] the nilpotent approximation $(M^\infty,\mathcal X^\infty)$ admits a  {\em Carnot group lifting} $G$ with {\em projection} $\pi^\infty:G\to M^\infty$ (see \cite[Definition~4.2]{MPVconotangente}) and the following holds: if $\bar \kappa$ is a horizontal lift of $\kappa $  to $G$ (i.e., $\kappa=\pi^\infty\circ\bar\kappa$), then $\bar\kappa$ is length-minimizing in $G$ and $\kappa$ is a horizontal line in $M^\infty$ if $\bar\kappa$ is a horizontal line in $G$, see \cite[Proposition~4.4]{MPVconotangente};
\item[(iv)] the projection $\pi^\infty$ maps $\Tan(\bar\kappa;0)$ into
$\Tan(\kappa;0)$,  see \cite[Proposition~4.3]{MPVconotangente}.
\end{itemize}
Hence, we are left to prove Theorem \ref{1.1} for $G$.
We can also assume that $t=0$ and $\gamma(0)=0$.

Let $\eta_i\downarrow 0$ be the sequence provided by Theorem
\ref{1.2}.
Since $\exc(\gamma;[-\eta_i,\eta_i])\to 0$, we can find a sequence
$\zeta_i\downarrow 0$ satisfying
        \[ \zeta_i^{-1/2}\exc(\gamma;[-\eta_i,\eta_i])\to 0. \]
        Let us set $\lambda_i:=\zeta_i\eta_i\downarrow 0$. Up to subsequences,
using Lemma \ref{cptmingrp} and a diagonal argument,     we can assume that
there exists a length minimizer $\gamma_\infty:\R\to G$ parametrized by
arclength such that
        \[
         \gamma_i(t):=\delta_{\lambda_i^{-1}}(\gamma(\lambda_i
t))\to\gamma_\infty(t) , \]
        uniformly on compact subsets of $\R$,     and that
$\underline{\dgamma_i}\to\underline{\dgamma_\infty}$
        in $L^2_{loc}(\R)$. For any fixed $N>0$ we have by Remark
\ref{rem:exctrdil}
        \[ \exc\pa{\gamma_i;[-N,N]}
=\exc(\gamma;[-N\zeta_i\eta_i,N\zeta_i\eta_i])\le(N\zeta_i)^{-1/2}\exc(\gamma;[
-\eta_i,\eta_i])\to 0, \]
        the last inequality being true for any $i$ such that $N\zeta_i\le 1$.
        So, by Remark \ref{exciscont}, we deduce that
$\exc(\gamma_\infty;[-N,N])=0$,
        which means that $\underline{\dgamma_\infty}(t)$
is contained in a hyperplane $\mathfrak h_1$ of $\g_1$ for a.e.~$t\in[-N,N]$.
Since this is true for any $N$, we deduce that there exists a hyperplane
$\mathfrak h_1$ of $\g_1$ such that $\underline{\dgamma_\infty}(t)\in \mathfrak h_1$ for
a.e.~$t\in\R$; in particular, $\gamma_\infty$ is contained in the  Carnot
subgroup $H$   associated with the Lie algebra generated by $\mathfrak h_1$.

        If the rank of $G$ is $r=2$, we conclude that $\gamma_\infty$ is
contained in a one-parameter subgroup of $\G$. Since  $\gamma_\infty\in
\Tan(\gamma;0)$ is a  length minimizer parametrized by arclength, we deduce that
$\gamma_\infty$ is a line in $G$.

        Otherwise, we can reason by induction on $r>2$: since $H$ has
rank $r-1$ and $\gamma_\infty$ is a  length minimizer in $H$ parametrized by
arclength, there exists $\widehat \gamma\in \Tan(\gamma_\infty;0)$ such that
$\widehat\gamma$ is a line in $H\subset G$. By \cite[Proposition 3.7]{MPVconotangente} we
have  $\widehat\gamma\in \Tan(\gamma;0)$ and the proof is accomplished.
        \end{proof}

        We state without proof the following version of Theorem
\ref{1.1}, which holds for extremal points of length minimizers; we refer to \cite[Section~3]{MPVconotangente} for the definitions of the {\em one-sided} tangent cones $\Tan^+(\gamma;0)$ and $\Tan^-(\gamma;T)$ of a horizontal curve $\gamma:[0,T]\to M$. The proof uses
the same arguments as the previous one and  can be easily deduced from Theorem
\ref{onesthm}.

\begin{theorem}\label{thm:1.1onesided}
    Let $\gamma:[0,T]\to M$ be a  length minimizer parametrized by arclength in
a Carnot--Carath\'eodory
    space $(M,d)$. Then, each of the tangent cones $\Tan^+(\gamma;0)$ and
$\Tan^-(\gamma;T)$ contains a horizontal
    half-line.
\end{theorem}

\begin{remark}\label{contreform}
In view of \cite[Remark~3.10]{MPVconotangente}, Theorem \ref{1.1} can be equivalently stated as follows: let $\gamma:[-T,T]\to M$ be a  length minimizer in $(M,d)$ parametrized by arclength and let $h\in L^\infty(-T,T)$ denote the controls of $\gamma$; then, for any $t\in(-T,T)$, there exist an infinitesimal sequence $\eta_i\downarrow 0$ and a constant unit vector $v\in S^{r-1}$ such that
\[ 
h(t+\eta_i\,\cdot)\to v\qquad\text{in }L^2_{loc}(\R). 
\]
Of course, an analogous version holds for extremal points.
\end{remark}

\end{document}